\documentclass[11pt,leqno]{amsart}
\usepackage{amsmath}
\usepackage{amssymb}

\def\underset#1#2{{\mathrel{\mathop {{}_{} {#2}}\limits_{{#1}_{}}}}}
\def\upplim_#1{\underset{#1}{\overline\lim}\;}
\def\lowlim_#1{\underset{#1}{\underline\lim}\;}
\newtheorem{corollary}[equation]{Corollary}

\newtheorem{lemma}[equation]{Lemma}
\newtheorem{proposition}[equation]{Proposition}

\newtheorem{theorem}[equation]{Theorem}

\newcommand{\C}{{\mathbb{C}}}

\renewcommand{\u}{{\bf u}}
\renewcommand{\a}{{\bf a}}
\renewcommand{\c}{{\bf c}}
\newcommand{\x}{{\bf x}}
\newcommand{\y}{{\bf y}}
\renewcommand{\P}{{\mathbb{P}}}
\newcommand{\Q}{{\mathbb{Q}}}
\newcommand{\R}{{\mathbb{R}}}
\newcommand{\Z}{\mathbb{Z}}

\numberwithin{equation}{section}

\title[Quantitative subspace theorem and general form of second main theorem]{Quantitative subspace theorem and general form of second main theorem for higher degree polynomials} 

\author{Si Duc Quang}

\begin{document}

\maketitle 

\begin{abstract}
This paper deals with the quantitative Schmidt's subspace theorem and the general from of the second main theorem, which are two correspondence objects in Diophantine approximation theory and Nevanlinna theory. In this paper, we give a new below bound for Chow weight of projective varieties defined over a number field. Then, we apply it to prove a quantitative version of Schmidt's subspace theorem for polynomials of higher degree in subgeneral position with respect to a projective variety. Finally, we apply this new below bound for Chow weight to establish a general form of second main theorem in Nevanlinna theory for meromorphic mappings into projective varieties intersecting hypersurfaces in subgeneral position with a short proof. Our results improve and generalize the previous results in these directions. 
\end{abstract}

\def\thefootnote{\empty}
\footnotetext{2010 Mathematics Subject Classification:
Primary 11J68, 32H30; Secondary 11J25, 11J97, 32A22.\\
\hskip8pt Key words and phrases: Diophantine approximation; subspace theorem; homogeneous polynomial; Nevanlinna theory; second main theorem; meromorphic mapping; hypersurface, subgeneral position.}

\section{Introduction and main result}
As we known, Diophantine approximation and Nevanlinna theory and have a close relation due to the works of Osgood (see \cite{Os81, Os85}) and Vojta (see \cite{V87}). By the dictionary of Vojta for the correspondences between these two theories, many results from Diophantine approximation may be translated into corresponding results from Nevanlinna theory and vice versa. In that way, the subspace theorem in Diophantine approximation corresponds to the second main theorem  (SMT) in Nevanlinna theory. Especially, the quantitative subspace theorem is very similar to the general form of the second main theorem. Over the last decades much research on these two theorems has been done. To state some of them, we recall the following.

Firstly, we introduce some notion in the side of Diophantine approximation. Throughout this paper, any given number field is assumed to be contained in a given algebraic closure $\overline{\mathbb Q}$. Let $K$ be a number field and denoted by $G_K$ the Galois group of ${\overline{\Q}}$ over $K$. For each ${\bf x}=(x_0,\ldots ,x_N)\in {\overline{\Q}}^{N+1},\sigma\in G_K$ we write $\sigma ({\bf x})=(\sigma (x_0),\ldots ,\sigma(x_N))$. Denote by $M_K$ the set of places (equivalence classes of absolute values) of $K$ and by $M^\infty_K$ the set of archimedean places of $K$. For each $v\in M_K$, we choose the normalized absolute value $|.|_v$ such that $|.|_v=|.|$ on $\mathbb Q$ (the standard absolute value) if $v$ is archimedean, and $|p|_v=p^{-1}$ if $v$ is non-archimedean and lies above the rational prime $p$. For each $v\in M_K$, denote by $K_v$ the completion of $K$ with respect to $v$ and set
 $$n_v := [K_v :\mathbb Q_v]/[K :\mathbb Q].$$
Putting $||x||_v=|x|^{n_v}_v$, we have the following product formula
$$\prod_{v\in M_K}||x||_v=1,\text{ for }  x\in K^*.$$
For ${\bf x} = (x_0 ,\ldots , x_N)\in K^{N+1}$, define
$$||{\bf x}||_v :=\max\{||x_0||_v,\ldots,||x_N||_v\},\ v\in M_K.$$
We define the absolute logarithmic height of a projective point ${\bf x}=(x_0:\cdots :x_N)\in\P^N(k)$ by
$$h({\bf x}):=\sum_{v\in M_K}\log ||{\bf x}||_v.$$
By the product formula, this definition does not depend on the choice of homogeneous coordinates $(x_0:\cdots :x_N)$ of ${\bf x}$. If $x\in K^*$, we define the absolute logarithmic height of $x$ by
$$h(x):=\sum_{v\in M_K}\log^+ ||x||_v,$$
where $\log^+a=\log\max\{1,a\}.$

For each $v\in M_K$, we choose an extension of $|.|_v$ to ${\overline{\Q}}$ which amounts to the extending $|.|_v$ to the algebraic closure $\overline{K}_v$ of $K_v$ and choosing an embedding ${\overline{\Q}}$ into $\overline{K}_v$.  For ${\bf x}=(x_0,\ldots,x_N)\in{\overline{\Q}}^{N+1}$ we put $||{\bf x}||_v:=\max\{|x_0|_v,\ldots,|x_N|_v\}$. For a given system $f_0,\ldots,f_m$ of homogeneous polynomials in ${\overline{\Q}}[x_0,\ldots,x_N]$, define $h(f_0,\ldots,f_m):=h({\bf a})$ where ${\bf a}$ is a vector consisting of the nonzero coefficients of $f_0,\ldots,f_m$. We denote by $K(f_0,\ldots,f_m)$ the extension of $K$ generated by the coefficients of $f_0,\ldots,f_m$. The height of a projective subvariety $X$ of $\P^N$ defined over ${\overline{\Q}}$ is defined by 
$$h(X):=h(F_X),$$
where $F_X$ is the Chow form of $X$ (see Section $\S $2).

With the above notations, the subspace theorem may be stated as follows. Let $n$ be a positive integer and $\delta$ be a real, $0<\delta<1$. Let $S$ be a finite set of places of $K$. For each $v\in S$, let $L^{(v)}_0,...,L^{(v)}_n$ be $n+1$ independent linear forms in $\Q[x_0,...,x_n]$. Then the set of solutions of 
$$ \log\left(\prod_{v\in S}\prod_{i=0}^n\frac{|L_i^{v}(\x)|_v}{||x||_v}\right)\le -(n+1+\delta)h(\x)$$
is contained in a finitely many proper linear subspace of $\P^n$.

The first results on the subspace theorem belong to Schmidt (see \cite{Sch72,Sch75,Sch89}). He also initially established a quantitative version of the subspace theorem by giving an explicit upper bound for the number of subspaces which contain all solutions with height large enough. After that, his results have been improved and generalized by many authors. For instance, in 2002, Evertse and Schlickewei (see \cite{EveSch02}, Theorem 3.1) deduced a quantitative version of the Absolute Subspace Theorem, dealing with solutions in $\P^{N}({\overline{\Q}})$ and linear forms. In 2008, Evertse and Ferretti (see \cite{EveFer08}, Theorem 1.3) generalized that result by considering higher degree polynomials instead of linear forms and the solutions being taken from a subvariety of $\P^N$, where $N\ge n\ge 1$. Their result is stated as follows.

\vskip0.2cm
\noindent
{\bf Theorem A} (Evertse and Ferretti \cite{EveFer08}, Theorem 1.3). {\it Let $\delta$ be a real with $0<\delta<1$, $K$ be a number field, $S$ be a finite set of places of $K$ of cardinality $s$, $X$ be a projective subvariety of $\P^N$ defined over $K$ of dimension $n\ge 1$ and of degree $d$, and $f^{(v)}_0,\ldots,f^{(v)}_n\ (v\in S)$ be a systems of homogeneous polynomials in ${\overline{\Q}}[x_0,\ldots, x_N]$. Put
\begin{align*}
\begin{cases}
&C :=\max([K(f^{(v)}_i ):K],v\in S,i=0,\ldots,n),\\
&\Delta := \mathrm{lcm}(\deg f^{(v)}_i\ :\ v\in S,i=0,\ldots,n), \text{(the least common multiple)}\\
&A_1:=(20n\delta^{-1})(n+1)s.\mathrm{exp}(2^{12n+16}n^{4n}\delta^{-2n}d^{2n+2}\Delta^{n(2n+2)}).\log{(4C)}\log\log{(4C)},\\
&A_2 :=(8n + 6)(n + 2)^2d\Delta^{n+1}\delta^{-1},\\
&A_3 :=\mathrm{exp}(2^{6n+20}n^{2n+3}\delta^{-n-1}d^{n+2}\Delta^{n(n+2)}\log{(2Cs)}),\\
&H :=\log (2N)+h(X)+\max(h(1,f^{(v)}_i), v\in S,0\le  i\le n).
\end{cases}
\end{align*}
Assume that
\begin{align}\label{1.1}
X({\overline{\Q}})\cap\{f^{(v)}_0=0,\ldots,f^{(v)}_n=0\}=\emptyset\text{ for }  v\in S.
\end{align}
Then there are homogeneous polynomials $G_1,\ldots,G_u\in K[x_0,\ldots,x_N]$ with
$$u\le A_1, \deg G_i\le A_2\text{ for }i=1,\ldots,u$$
which do not vanish identically on $X$, such that the set of $\x\in X({\overline{\Q}})$ with
\begin{align}
\label{1.2}
\log\left(\prod_{v\in S}\prod_{i=0}^n\max\limits_{\sigma\in G_K}\frac{|f^{v}_i(\sigma(\x))|_v^{1/\deg f^{(v)}_i}}{||\sigma(\x)||_v}\right)\le -(n+1+\delta)h(\x),
\end{align}
$$ h(\x)\ge A_3H $$
is contained in $\bigcup_{i=1}^u(X\cap\{G_i=0\})$.}

We note that, the method of Evertse and Ferretti to prove the above theorem is based on the below bound of Chow weight for projective varieties. Recently, by combining this method of Evertse and Ferretti with the technique of Chen, Ru and Yan in \cite{CRY02} on filtrating the homogeneous polynomial vector space, L. Giang  in \cite{LG} has given an extension of the above theorem for the case when the condition (\ref{1.1}) is replaced by a more generalized condition
\begin{align}\label{1.3}
X({\overline{\Q}})\cap\{f^{(v)}_0=0,\ldots,f^{(v)}_m=0\}=\emptyset\text{ for }0\le i\le m, v\in S.
\end{align}
Unfortunately, in the result of L. Giang the right hand side of the inequality (\ref{1.2}) is replaced by $-\left((m+1)(n+1)+\delta\right)h(\x)$. Therefore, when $m=n$, her result does not comeback to the original result of Everste and Ferretti.
Also, we note that in \cite{LG}, Giang need an additional assumption that
$$X({\overline{\Q}})\not\subset\{f^{(v)}_i = 0\}\text{ for }i=0,\ldots, m,v\in S,$$
Then, a natural question arising here is that: \textit{``How to generalize Theorem A when the condition (\ref{1.1}) is replaced by (\ref{1.3}) and overcome these restrictions.''}.

One of the main our purpose in this paper is to give a positive answer for the above question. To do so, firstly we will give a new below bound for Chow weight of a projective variety as follows.

\vskip0.2cm
\noindent
\begin{theorem}\label{1.4}
Let $Y$ be a projective subvariety of $\P^R$ of dimension $n\ge 1$ and degree $D$, defined over ${\overline{\Q}}$. Let $m\ (m\ge n)$ be an integer and let ${\bf c}=(c_0,\ldots,c_R)$ be a tuple of nonnegative reals. Let $\{i_0,\ldots, i_m\}$ be a subset of $\{0,\ldots,R\}$ such that
$$Y({\overline{\Q}})\cap\{y_{i_0}=0,\ldots,y_{i_m}=0\}=\emptyset.$$
Then
$$e_Y({\bf c})\ge \frac{D}{m-n+1}(c_{i_0}+\cdots+c_{i_m}).$$
\end{theorem}
%We also note that, in this theorem, we consider the Chow weight of a projective subvariety $Y$ of dimension $n$ of $\P^R$, which has no common point with the intersection of $(m+1)$ coordinates hyperplanes. 
We see that if $m=n$, the above theorem will cover the previous result of Everste and Ferretti (see \cite[Proposition 4.1]{EveFer08}), which is the key for their proof of the quantitative subspace theorem. 
Using this result, we will give a quantitative subspace theorem as follows.
\vskip0.2cm
\noindent
\begin{theorem}\label{1.5}
Let $\delta$ be a real with $0<\delta<1$, $K$ be a number field, $S$ be a finite set of places of $K$ of cardinality $s$, $X$ be a projective subvariety of $\P^N$ defined over $K$ of dimension $n\ge 1$ and of degree $d$, and $f^{(v)}_0,\ldots,f^{(v)}_m\ (v\in S)$ be a systems of homogeneous polynomials in ${\overline{\Q}}[x_0,\ldots, x_N]$, where $m\ge n$. Put
\begin{align}
\label{1.6}
&\begin{cases}
&C :=\max([K(f^{(v)}_i ):K],v\in S,i=0,\ldots,n),\\
&\Delta := \mathrm{lcm}(\deg f^{(v)}_i\ :\ v\in S,i=0,\ldots,n),
\end{cases}\\
\label{1.7}
&\begin{cases}
&A_1:=(18n\delta^{-1})(n+1)s.\mathrm{exp}(2^{12n+16}(m-n+1)^{4n}n^{4n}\delta^{-2n}d^{2n+2}\Delta^{n(2n+2)})\\
&\ \ \ \ .\log{(4C)}\log\log{(4C)},\\
&A_2 :=(8n + 6)(n + 2)^2d\Delta^{n+1}\delta^{-1},\\
&A_3 :=\mathrm{exp}(2^{6n+20}n^{2n+3}\delta^{-n-1}d^{n+2}\Delta^{n(n+2)}\log{(2Cs)}),\\
&H :=\log (2N)+h(X)+\max(h(1,f^{(v)}_i), v\in S,0\le  i\le n).
\end{cases}
\end{align}
Assume that
\begin{align}\label{1.8}
X({\overline{\Q}})\cap\{f^{(v)}_0=0,\ldots,f^{(v)}_m=0\}=\emptyset\text{ for }  v\in S.
\end{align}
Then there are homogeneous polynomials $G_1,\ldots,G_u\in K[x_0,\ldots,x_N]$ with
$$u\le A_1, \deg G_i\le A_2\text{ for }i=1,\ldots,u$$
which do not vanish identically on $X$, such that the set of $\x\in X({\overline{\Q}})$ with
\begin{align}
\label{1.9}
&\log\left(\prod_{v\in S}\prod_{i=0}^n\max\limits_{\sigma\in G_K}\frac{|f^{v}_i(\sigma(\x))|_v^{1/\deg f^{(v)}_i}}{||\sigma(\x)||_v}\right)\le -\left((m-n+1)(n+1)+\delta\right)h(\x),\\
\label{1.10}
&h(x)\ge A_3H
\end{align}
is contained in $\bigcup_{i=1}^u(X\cap\{G_i=0\})$.
\end{theorem}

\vskip0.2cm
Then, we see that our result will implies Theorem A when $m=n$.  Also, since $\P^N(K)$ has only finitely many points with height bounded below by a certain value, the above theorem immediately implies the following corollary.

\begin{corollary}\label{1.11}
Let $\delta$ be a positive number with $\delta\le 1$, $K$ be a number field, $S$ be a finite set of places of $K$ of cardinality $s$, $X$ be a projective subvariety of $\P^N$ defined over $K$ of dimension $n\ge 1$ and of degree $d$. Let $m$ be an integer with $m\ge n$ and $f^{(v)}_0,\ldots, f^{(v)}_m\ (v\in S)$ be systems of homogeneous polynomials in ${\overline{\Q}}[x_0,\ldots, x_N]$such that
$$X({\overline{\Q}})\cap\{f^{(v)}_0=0,\ldots,f^{(v)}_m=0\}=\emptyset\text{ for } v\in S.$$
Then, the set of solutions $x\in X(K)$ of the inequality
$$ 
\log\left(\prod_{v\in S}\prod_{i=0}^n\frac{|f^{v}_i(\sigma(\x))|_v^{1/\deg f^{(v)}_i}}{||\sigma(\x)||_v}\right)\le -\left((m-n+1)(n+1)+\delta\right)h(\x),
$$
is contained in a finite union of proper subvarieties of $X$.
\end{corollary}

On the other side, motivated by Schmidt's subspace theorem in number theory, Vojta \cite{V97} gave the following general form of the Second Main Theorem for holomorphic curves. With the standard notation in Nevanlinna theory, his result is stated as follows.

\vskip0.2cm
\noindent
{\bf Theorem B} (see \cite[Theorem 1]{V97})\ {\it Let $H_1,\ldots, H_q$ be $q$ arbitrary hyperplanes in $\P^N(\C)$. Let $f:\C\rightarrow\P^n(\C)$ be a linearly nondegenerate holomorphic curve. Let $(f_0:\cdots:f_N)$ be coordinate functions for $f$, chosen to be holomorphic, and let $W$ denote the Wronskian of $(f_0,\ldots, f_n)$. Then
$$||\ \int\limits_{S(r)}\max_{K}\log\prod_{j\in K}\frac{||f(x)||\cdot ||H_j||}{|H_j(f(x))|}\sigma_p\le (N+1)T_f(r)-N_{W}(r)+S_f(r),$$
where the sum is taken over all subsets $K$ of $\{1,...,q\}$ such that $\{H_j\ ;\ j\in K\}$ is linear independent.}

From Theorem B of Vojta, we easily get the second main theorem for a linearly non-degenerate meromorphic mapping $f$ intersecting $q\ge N+2$ hyperplanes in general position of $\P^n(\C)$ as the following form
$$||\ (q-n-1)T_f(r)\le\sum_{i=1}^qN^{[N]}_{H_i(f)}(r)+S_f(r). $$
Here, by the notation ``$||\ P$'' we mean the assertion $P$ holds for all $r\in [0;+\infty)$ except a Borel finite measure set and $S_f(r)$ stands for a small term with respect to $T_f(r)$.
 
Theorem B also is improved by Min Ru \cite{Ru97} with a better small term. In \cite{Ru97}, M. Ru also shows that the general form of second main theorem for fixed  hyperplanes may be applied to give a simple proof for second main theorem for moving hyperplanes.

In the last century, the study on SMT for hyperplanes (fixed or moving hyperplanes in either general or subgeneral position) of projective space have been almostly completed by the works of Cartan \cite{Ca}, Nochka \cite{Noc83}, Stoll-Ru \cite{RS01,RS02}, Shirosaki \cite{Sh90}, Noguchi \cite{No05} and others. However, the SMT for hypersurfaces has just been studied in some recent years together with the development of the Diophantine approximation theory. Firstly, by using the method of Zanier and Corvaja \cite{CZ} on filtralation of vector spaces of homogeneous polynomials, Ru \cite{Ru04} established SMT for meromorphic mappings into projective spaces intersecting hypersurfaces in general position. After that, adopting the method of Feratti and Evertse \cite{EveFer08} on using Chow weight, Ru \cite{Ru09} proved the SMT for the case the mappings into projective varieties intersecting hypersurfaces in general position. 

Beside the most important question on the truncation level of the counting functions, a remain open question in this topic is to study SMT for the case where the hypersurfaces in subngeneral position. Since there is no Nochka's weights for the case of hypersurfaces in subgeneral position (a main tool to prove SMT for targets in subgeneral position), almost all SMTs in this case available at present are still not yet optimal (we refer reader to the works of Z. Chen, M. Ru and Q. Yan \cite{CRY01,CRY02}, L. Giang \cite{LG02}, S. Lei and M. Ru \cite{LR}). In particular they cannot deduce the SMT for the case of hypersurfaces in general position.

In 2017, Quang \cite{Q2017} introduced the method ``replacing hypersurfaces''. In general, his method allows that an SMT for the case of targets in general position can be generalized to an SMT for the case of targets in subgeneral position without using Nochka's weight. Using this method and slightly developing the method of Evertse and  Feratti \cite{EveFer08}, Quang had given an SMT for meromorphic mappings into projective varieties intersecting hypersurfaces in subgeneral position with a quite complicated proof (see Theorem 1.1 in \cite{Q2017}).

In the last part of this paper, using the new below bound of  Chow weight in Theorem \ref{1.4}, we will prove a general form of SMT for the mappings into projective varieties intersecting hypersurfaces in subgeneral postion, and then apply it to get again version non trucated SMT of Quang in \cite{Q2017}. Our proof in this paper will be more simpler than that in \cite{Q2017}. For detail, we will proof the following general form of SMT.

\begin{theorem}\label{1.12}
Let $V\subset\P^N(\C)$ be a smooth complex projective variety of dimension $n\ge 1$. Let $f:\C^p\to V$ be an algebraically nondegenerate meromorphic mapping.
Let $Q_1,\ldots, Q_q$ be $q$ arbitrary hypersurfaces in $\P^N(\C)$ the same degree $\Delta$ and let $m\ge n$ be an integer. Then, for every $\epsilon >0$,
$$||\ \int\limits_{S(r)}\max_{K\subset\mathcal K}\log\prod_{j\in K}\frac{||f(x)||^{\Delta}\cdot ||Q_j||}{|Q_j(f(x))|}\sigma_p\le \left (\Delta(m-n+1)(n+1)+\epsilon\right)T_f(r),$$
where $\mathcal K$ denotes the set of all subsets $K$ of $\{1,...,q\}$ with $\dim\left (V\cap\bigcap_{j\in K}Q_j\right)\le m-\sharp K$.
\end{theorem}

With the routine argument of Nevanlinna theory, the above result immediately implies the following second main theorem.
\begin{corollary}\label{1.13}
Let $V\subset\P^N(\C)$ be a smooth complex projective variety of dimension $n\ge 1$. Let $Q_1,...,Q_q$ be hypersurfaces in $\P^N(\C)$ of degree $d_i$, located in $m-$subgeneral position with respect to $V$. Let $f:\C^p\to V$ be an algebraically nondegenerate meromorphic mapping. Then, for every $\epsilon >0$, 
$$\bigl |\bigl |\ \left (q-(m-n+1)(n+1)-\epsilon\right) T_f(r)\le\sum_{i=1}^q\frac{1}{d_i}N_{Q_i(f)}(r).$$
\end{corollary} 
Then, we get again the version non truncated second main theorem for meromorphic mappings intersecting hypersurfaces in subgeneral position. 
We note that the general form of SMT will immediately implies the SMT but not vice versa. 

Finally, we would also like to propose the following question: ``\textit{Is there any way to prove SMT for moving hypersurfaces in subgeneral position by applying Theorem \ref{1.2} as Min Ru did for the case of moving hyperplanes in \cite{Ru97}}?''

If the answer of the above question is affirmative then we hope to find a simple proof for SMT for moving hypersurfaces in subgeneral position.

\section{Below Bound for Chow wieght}
In this section, we will give the proof for Theorem \ref{1.4}. Firstly, we recall some following.

\vskip0.2cm
\noindent
\textbf{2.1. Chow form and Chow weight.}

Let $K$ be a number field or $K=\C$. Let $X\subset\P^N$ be a projective subvariety of dimension $n$ and degree $D$, defined over $K$.  The Chow form associated to $X$ is a polynomial
$$F_X(\textbf{u}_0,\ldots,\textbf{u}_n) = F_X(u_{00},\ldots,u_{0N};\ldots; u_{n0},\ldots,u_{nN})$$
in $n+1$ blocks of variables $\textbf{u}_i=(u_{i0},\ldots,u_{iN}), i = 0,\ldots,n,$ which satisfies the following
properties: 
\begin{itemize}
\item $F_X$ is irreducible in ${\overline{\Q}}[u_{00},\ldots,u_{nN}]$,
\item $F_X$ is homogeneous of degree $D$ in each block $\textbf{u}_i, i=0,\ldots,n$,
\item $F_X(\textbf{u}_0,\ldots,\textbf{u}_n) = 0$ if and only if $X({\overline{\Q}})\cap H_{\textbf{u}_0}\cap\cdots\cap H_{\textbf{u}_n}\ne\varnothing$, where $H_{\textbf{u}_i}, i = 0,\ldots,n$, are the hyperplanes given by $u_{i0}x_0+\cdots+ u_{iN}x_N=0.$
\end{itemize}
Let ${\bf c}=(c_0,\ldots, c_N)$ be a tuple of real numbers. Let $t$ be an auxiliary variable. We consider the decomposition
\begin{align}\label{2.1}
\begin{split}
F_X(t^{c_0}u_{00},&\ldots,t^{c_N}u_{0N};\ldots ; t^{c_0}u_{n0},\ldots,t^{c_m}u_{nm})\\ 
& = t^{e_0}G_0(\textbf{u}_0,\ldots,\textbf{u}_m)+\cdots +t^{e_r}G_r(\textbf{u}_0,\ldots, \textbf{u}_m).
\end{split}
\end{align}
with $G_0,\ldots,G_r\in K[u_{00},\ldots,u_{0m};\ldots; u_{n0},\ldots,u_{nm}]$ and $e_0>e_1>\cdots>e_r$. The Chow weight of $X$ with respect to ${\bf c}$ is defined by
\begin{align}\label{2.2}
e_X({\c}):=e_0.
\end{align}
The following theorem is due to J. Evertse and R. Ferretti \cite{EveFer02}.
\begin{theorem}[{\cite[Theorem 4.1]{EveFer02}, see also \cite[Theorem 2.1]{Ru09}}]\label{2.12new}
Let $X\subset\P^N$ be an algebraic variety of dimension $n$ and degree $D$ defined over the field $K$. Let $u>D$ be an integer and let ${\bf c}=(c_0,\ldots,c_N)\in\mathbb R^{N+1}_{\geqslant 0}$.
Then
$$ \frac{1}{uH_X(u)}S_X(u,{\bf c})\ge\frac{1}{(N+1)D}e_X({\bf c})-\frac{(2N+1)D}{u}\cdot\left (\max_{i=0,...,N}c_i\right). $$
\end{theorem}
\noindent
\textbf{2.2. Proof of Theorem \ref{1.4}.} 

Without loss of generality, we may suppose that $i_0=0,...,i_m=m$ and 
$$c_0\ge c_1\ge\cdots\ge c_m.$$
It is easy to see that 
$$\mathrm{dim}\left(Y({\overline{\Q}})\cap\{y_0=\cdots=y_k=0\}\right)\le m-k-1\ \text{ for all } k=m-n+1,...,m.$$

Using the idea ``replacing hypersurfaces'' from \cite{Q2017}, we will construct $n$ new affine coordinate functions of $\P^R(\C)$ as follows:

\noindent
$\bullet$ Step 1. Set $y_0'=y_0$. We also set $c_{00}=1$ and $c_{0k}=0$ for all $0<k\le R$.

\noindent
$\bullet$ Step 2. For each irreducible component $\Gamma$ of dimension $(n-1)$ of $Y({\overline{\Q}})\cap\{y_0'=0\}$, set 
$$V_{1\Gamma}=\left\{c=(c_1,...,c_{m-n+1})\in {\Q}^{m-n+1}\ ;\ \Gamma\subset \{c_1y_1+\cdots c_{m-n+1}y_{m-n+1}=0\}\right\}.$$
Then $V_{1\Gamma}$ is a linear subspace of ${\Q}^{m-n+1}$. Since $Y({\overline{\Q}})\cap\{y_0'=0\}\cap\{y_1=\cdots=y_{m-n+1}=0\}$ has the dimension of at most $n-2$, there exists $i\in\{1,...,m-n+1\}$ such that $\Gamma\not\subset \{y_i=0\}$. Therefore $V_{1\Gamma}$ is a proper linear subspace of ${\Q}^{m-n+1}$. Since there are only finite many irreducible components with dimension $(n-1)$ of $Y({\overline{\Q}})\cap\{y_0'=0\}$, 
$${\Q}^{m-n+1}\setminus\bigcup_{\Gamma}V_{1\Gamma}\ne\emptyset. $$
Hence, there exists $(c_{11},...,c_{1(m-n+1)})\in {\Q}^{m-n+1}$ such that
$$\Gamma\not\subset \{c_{11}y_1+\cdots+c_{1(m-n+1)}y_{m-n+1}=0\}$$
for every irreducible component $\Gamma$ of dimension $(n-1)$ of $Y({\overline{\Q}})\cap\{y_0'=0\}$. We set
$$ y_1'=c_{11}y_1+\cdots+c_{1(m-n+1)}y_{m-n+1}.$$
This clearly implies that $\dim Y({\overline{\Q}})\cap\{y_0'=y_1'=0\}\le n-2.$
We also set $c_{10}=0$ and $c_{1k}=0$ for all $m-n+1<k\le R$.

\noindent
$\bullet$ Step 3. For every irreducible component $\Gamma'$ of dimension $(n-2)$ of $Y({\overline{\Q}})\cap\{y_0'=y_1'=0\}$, we set 
$$V_{2\Gamma'}=\left\{c=(c_1,...,c_{m-n+2})\in {\Q}^{m-n+2}\ ;\ \Gamma\subset \{c_1y_1+\cdots c_{m-n+2}y_{m-n+2}=0\}\right\}.$$
Hence, $V_{2\Gamma'}$ is a linear subspace of ${\Q}^{m-n+2}$. Since $Y({\overline{\Q}})\cap\{y_0'=y'_1=0\}\cap\{y_1=\cdots=y_{m-n+2}=0\}$ has the dimension of at most $n-3$, there exists $i\in\{1,...,m-n+2\}$ such that $\Gamma\not\subset \{y_i=0\}$. Therefore $V_{2\Gamma'}$ is a proper linear subspace of ${\Q}^{m-n+2}$. Since there are only finite many irreducible components with dimension $(n-2)$ of $Y({\overline{\Q}})\cap\{y_0'=y_1'=0\}$, 
$${\Q}^{m-n+2}\setminus\bigcup_{\Gamma'}V_{2\Gamma'}\ne\emptyset. $$
Hence, there exists $(c_{21},...,c_{2(m-n+2)})\in {\Q}^{m-n+2}$ such that
$$\Gamma\not\subset \{c_{21}y_1+\cdots+c_{2(m-n+1)}y_{m-n+2}=0\}$$
for all irreducible components of dimension $(n-2)$ of $Y({\overline{\Q}})\cap\{y_0'=y_1'=0\}$. We set
$$ y_2'=c_{21}y_1+\cdots+c_{2(m-n+2)}y_{m-n+2}.$$
This clearly implies that $\dim Y({\overline{\Q}})\cap\{y_0'=y_1'=y_2'=0\}\le n-3.$
We also set $c_{20}=0$ and $c_{2k}=0$ for all $m-n+2<k\le R$.

Repeating again the above step, after the $n^{th}-$step we get $\c_i=(c_{i0},...,c_{iR})\ (0\le i\le n)$, $y_0',...,y_n'$ satisfying:
\begin{align*}
& c_{00}=1,c_{0k}=0\text{ for all } 0< k\le R,\\ 
& c_{i0}=0 \text{ and } c_{ik}=0\text{ for all }1\le i\le n\text{ and }m-n+i< k\le R,\\
& y_t'=c_{t0}y_0+\cdots +c_{tm}y_m\text{ for all }0\le t\le n.\\
& \dim Y({\overline{\Q}})\cap\{y_0'=\cdots=y_t'=0\}\le n-t-1, \ 0\le t\le n.
\end{align*}
In particular, 
\begin{align}\label{2.3}
Y({\overline{\Q}})\cap\{y_0'=\cdots=y_n'=0\}.
\end{align}

For each subset $I=\{k_0,\ldots,k_n\}$ of $\{0,\ldots,R\}$ with $k_0<k_1\cdots<k_n$, we define the bracket
$$ [I]=[I](\u^{(0)},\ldots,\u^{(n)}):=\det\left(u^{(i)}_{k_j}\right)_{0\le i,j\le n}, $$
where $\u^{(h)}=(u^{(h)}_0,\ldots,u^{(h)}_R)$. Let $I_1,...,I_s$ be all subsets of $\{0,\ldots,R\}$ of cardinality $(n+1)$, where $s=\binom{R+1}{n+1}$. Let $A$ be the set of tuples of nonnegagative integers $\a=(a_1,\ldots,a_s)$ with $a_1+\cdots+a_s=D$. Then the Chow form $F_Y$ can be written as a homogeneous polynomial of degree $D$ in $[I_1],\ldots,[I_s]$:
\begin{align}\label{2.4}
F_Y=\sum_{\a\in A}C(\a)[I_1]^{a_1}\ldots[I_s]^{a_s}, 
\end{align}
where $\a=(a_1,...,a_s)$ and $C(\a)\in A$ (see Theorem IV in \cite{Hog52}). We note that
$$ [I](t^{c_0}u^{(0)}_0,\ldots,t^{c_R}u_R^{(0)};\ldots;t^{c_0}u^{(n)}_0,\ldots,t^{c_R}u_R^{(n)})=t^{\sum_{i\in I}c_i}[I]. $$
Therefore, we have
\begin{align}\label{2.5}
\begin{split}
F_Y&(t^{c_0}u^{(0)}_0,\ldots,t^{c_R}u_R^{(0)};\ldots;t^{c_0}u^{(n)}_0,\ldots,t^{c_R}u_R^{(n)})\\ 
& =\sum_{\a\in A}C(\a)t^{\sum_{j=1}^sa_j(\sum_{i\in I_j}c_i)}[I_1]^{a_1}\cdots[I_s]^{a_s}.
\end{split}
\end{align}

Now, for each $I=\{k_0,\ldots,k_n\}\subset\{0,\ldots,R\}$ with $k_0<k_1\cdots<k_n$, we have
$$[I](\c_0,...,\c_n)=\det\left(c_{ik_j}\right)_{0\le i,j\le n}.$$
We see that if $k_0>0$ then all elements in the first row of the matrix $\left(c_{ik_j}\right)_{0\le i,j\le n}$ are zeros; if there exists an index $j\ge 1$ such that $k_j> m-n+j$ then $k_n\ge m$ and hence all elements in the $(n+1)^{th}$ comlumn of the matrix $\left(c_{ik_j}\right)_{0\le i,j\le n}$ are zeros. Therefore if $[I](\c_0,...,\c_n)\ne 0$ then $k_0=0$ and $k_j\le m-n+j$ for all $1\le j\le n$.

Now by (\ref{2.3}) we have $F_Y(\c_0,...,\c_n)\ne 0$. Hence in the expression (\ref{2.4}) there is a term $C[I_1]^{a_1}\cdots[I_s]^{a_s}$ with $C\in\overline{\Q}^*$, $a_1+\cdots+a_s=D$, $I_j=\{k_{j0},\ldots,k_{jn}\}\ (1\le j\le s)$ such that $k_{j0}=0<k_{j1}<\cdots <k_{jn}$ and $k_{jl}\le m-n+l$. We substitute $\u^{(i)}=\c_i\ (i=0,...,n)$ in (\ref{2.5}) and obtain $C\cdot t^{\sum_{j=1}^sa_j(\sum_{i\in I_j}c_i)}$. Therefore, one of the numbers $e_i$ in (\ref{2.1}) is equal to $\sum_{j=1}^sa_j(\sum_{i\in I_j}c_i)$. We have the following estimate
\begin{align*}
\sum_{j=1}^sa_j(\sum_{i\in I_j}c_i)&=\sum_{j=1}^sa_j(\sum_{i=0}^{n}c_{k_{ji}})\ge \sum_{j=1}^sa_j(c_0+\sum_{i=1}^nc_{m-n+i})=D(c_0+\sum_{i=1}^nc_{m-n+i})\\
&\ge\dfrac{D}{m-n+1}\left((m-n+1)c_0+\sum_{i=1}^nc_{m-n+i}\right)\ge\dfrac{D}{m-n+1}\sum_{i=0}^mc_i.
\end{align*}
 This completes the proof of the theorem.\hfill$\square$

\section{Quantitative Schmidt's subspace theorem}

In this section, we will proof Theorem \ref{1.5}. We separate this section into two parts. In part A, we will recall some notion and results from \cite{EveFer02,EveFer08}. The proof of Theorem \ref{1.5} will be included in Part B.

\vskip0.2cm
\noindent
\textbf{Part A. Twisted hight and Hilbert weight.}

\vskip0.2cm
\noindent
\textbf{3.1.} Let $K$ be a number field and $L$ be a finite extension of $K$. If $w$ is a place of $L$ which lies above a place $v$ of $K$, then
\begin{align}\label{3.1}
|x|_w=|x|^{d(w|v)}_v, \text{ for } x\in K\text{ with } d(w|v)=\frac{[L_w : K_v]}{[L:K]}.
\end{align}
For $v\in M_K$, let $\c_v=(c_{0v},\ldots, c_{Rv})$ be a tuple of reals such that $c_{0v}=\cdots=c_{Rv}=0$ for all but finitely many places $v\in M_K$ and put $\c=(c_v\ :\ v\in M_K)$. Further, let $Q\ge 1$ be a real. 
%We define the twisted height on $\P^R({\overline{\Q}})$ as follows. First put
For $\y=(y_0,\ldots,y_R)\in\P^R(K)$, we define
$$H_{Q,\c}(\y):=\prod_{v\in M_K}\max\limits_{0\le i\le R}\left (|y_i|_vQ^{c_{iv}}\right).$$
By the product formula, this is well-defined on $\P^R(K)$. For any finite extension $L$ of $K$ we put
\begin{align}\label{3.2}
c_{iw}:=c_{iv}.d(w|v) \text{ for } w\in M_L,
\end{align}
where $M_L$ is the set of places of $L$ and $v$ is the place of $K$ lying below $w.$ Then, for the general case when $\y\in\P^R({\overline{\Q}})$, we define its twisted height by
\begin{align}\label{3.3}
H_{Q,\c}(\y):=\prod_{w\in M_L}\max\limits_{0\le i\le R}\left(|y_i|_wQ^{c_{iw}}\right),
\end{align}
where $L$ is any finite extension of $K$ such that $y\in \P^R(L)$. By (\ref{3.1}) and (\ref{3.2}), we see that this definition does not depend on $L$.

\noindent
\textbf{3.2.} We recall some results of Everste and Ferreti \cite{EveFer08} for twisted heights as follows. Let $Y$ be a projective subvariety of $\P^R$ of dimension $n\ge 1$ and degree $D$, defined over $K$, and let $c_v=(c_{0v},\ldots,c_{Rv})\ (v\in M_K)$ be tuples of reals such that
\begin{align}
\label{3.4}
&c_{iv}\ge 0\text{ for }v\in M_K,i=0,\ldots , R,\\ 
\label{3.5}
&c_{0v}=\cdots=c_{Rv}=0 \text{ for all but finitely many }v\in M_K,\\
\label{3.6}
&\sum_{v\in M_K}\max\{c_{0,v},\ldots,c_{R,v}\}\le 1. 
\end{align}
Put
\begin{align}\label{3.7}
E_Y({\bf c}):=\frac{1}{(n+1)D}\left (\sum_{v\in M_K}e_Y({\bf c}_v)\right),
\end{align}
where $e_Y({|{\bf c}|}_v)$ is the Chow weight defined in Section 2.

Further, let $0<\delta\le 1$, and put
\begin{align}\label{3.8}
\begin{split}
\begin{cases}
B_1&:=\mathrm{exp}(2^{10n+4}\delta^{-2n}D^{2n+2}).\log (4R)\log\log (4R),\\
B_2&:=(4n+3)D\delta^{-1},\\
B_3&:=\exp(2^{5n+4}\delta^{-n-1}D^{n+2}\log (4R)).
\end{cases}
\end{split}
\end{align}
\begin{theorem}[see Theorem 2.1 \cite{EveFer08}]\label{3.9}
There are homogeneous polynomials $F_1,\ldots,F_t\in K[y_0,\ldots,y_R]$ with
$$t\le  B_1, \deg F_i\le  B_2 \text{ for }i = 1,\ldots,t,$$
which do not vanish identically on $Y,$ such that for every real number $Q$ with
$$\log Q \ge B_3.(h(Y)+1)$$
there is $F_i\in\{F_1,\ldots,F_t\}$ with
\begin{align}\label{3.10}
\left\{y\in Y({\overline{\Q}})\ :\ H_{Q,\c}({\bf y})\le Q^{E_Y({\bf c})-\delta}\right\}\subset Y\cap\{F_i=0\}.
\end{align}
\end{theorem}

\noindent
\textbf{3.3.} For each place $v\in M_K$, we choose a normalized absolute $|.|_v$ as in Section 1 and choose its an extension to ${\overline{\Q}}$. In particular, for each $v\in M_K^\infty$, there is an isomorphic embedding $\sigma_v:{\overline{\Q}}\rightarrow\C$ such that
$$|x|_v=|\sigma_v(x)|^{[K_v:\R]/[K:\Q]}\text{ for }x\in{\overline{\Q}}.$$
For a polynomial $f$, we write $f=\sum_{m\in M_f}c_f(m)m$, where the symbol $m$ denotes a monomial, $M_f$ is a finite set of monomials, and $c_f(m)\ (m\in M_f)$ are the coefficients. For any map $\sigma$ on the definition field of $f$, we put
$$ \sigma (f):=\sum_{m\in M_f}\sigma (c_f(m))\cdot m .$$
Let $f_i=\sum_{m\in M_{f_i}}c_{f_i}(m)\cdot m\ (i=1,\ldots,r)$ be $r$ polynomials with complex coefficients. We define the following norms:
$$||f_1,\ldots,f_r||:=\max\left (|c_{f_i}(m)|\ : 1\le i \le r,m\in M_{f_i}\right ),$$
$$||f_1,\ldots, f_r||_1:=\sum_{i=1}^r\sum_{m\in M_{f_i}}|c_{f_i}(m)|.$$
If all coefficients of $f_i\ (i=1,...,r)$ belong to ${\overline{\Q}}$, we define
\begin{align}\label{3.11}
\begin{split}
\begin{cases}
&||f_1,\ldots, f_r||_v:=\max\left(|c_{f_i}(m)|_v\ :\ 1\le i \le r,m\in M_{f_i}\right)\ (v\in M_K),\\
&||f_1,\ldots , f_r||_{v,1}:= ||\sigma_v(f_1),\ldots ,\sigma_v(f_r)||_1^{[K_v:R]/[K:Q]}\  (v\in M^\infty_K),\\
&||f_1,\ldots, f_r||_{v,1}:=||f_1,\ldots ,f_r||_v\ (v\in M^0_K).
\end{cases}
\end{split}
\end{align}
If all coefficients of $f_i\ (i=1,...,r)$ belong to $K$, we may define heights
$$h(f_1,\ldots, f_r):=\log\left(\prod_{v\in M_K}||f_1,\ldots, f_r||_v\right ),$$
$$h_1(f_1,\ldots, f_r):=\log\left(\prod_{v\in M_K}||f_1,\ldots, f_r||_{v,1}\right ),$$
More generally, for polynomials $f_1,\ldots,f_r$ with coefficients in ${\overline{\Q}}$, we  choose a number field $K$ containing the coefficients of $f_1,\ldots,f_r$ and define the weights $h(f_1,\ldots,f_r)$, $h_1(f_1,\ldots,f_r)$ as above. We see that this is definition independent of the choice of $K$.

We list here some following facts
\begin{itemize}
\item For $x\in{\overline{\Q}}^{N+1}$ and $f\in{\overline{\Q}}[x_0,\ldots, x_N]$ homogeneous of degree $D$, 
\begin{align}\label{3.12}
||f(x)||_v\le ||f||_{v,1}||x||^D_v\text{ for }v\in M_K.
\end{align}
\item For $x\in\P^N({\overline{\Q}})$ and $f_0,\ldots, f_r\in{\overline{\Q}}[x_0,\ldots, x_N]$ homogeneous of degree $D$, 
\begin{align}\label{3.13}
||f_0(x),\ldots, f_r(x)||\le Dh(x)+h_1(f_0,\ldots, f_r).
\end{align}
\end{itemize}

Now, we recall the following proposition from \cite[Proposition 4.5]{EveFer08}.
\begin{proposition}[see Proposition 4.5 \cite{EveFer08}]\label{3.14} Let $X$ be $n$-dimensional subvariety of $\P^N$ of degree $d$ defined over ${\overline{\Q}}.$ Let $g_0,\ldots, g_R$ be homogeneous polynomials of degree $\Delta$ in $ {\overline{\Q}}[x_0,\ldots, x_N]$ such that
$$X({\overline{\Q}})\cap{g_0=0,\ldots, g_R=0}=\emptyset.$$
Let $Y=\varphi (X)$, where is the morphism on $X$ given by $x\mapsto (g_0(x),\ldots, g_R(x))$. Then
\begin{align*}
h(Y ) \le& \Delta^{n+1}h(X)+(n+1)d\Delta^nh_1(g_0,\ldots ,g_R)\\
&+5(n+1)d\Delta^{n+1}\log(N+\Delta)+3(n+1)d\Delta^n\log(R + 1).
\end{align*}
\end{proposition}

We state here an useful lemma due to Evertse and Ferretti.
\begin{lemma}[ see \cite{EveFer08}, Lemma 5.2]\label{3.15}
Let $\theta$ be a real with $0<\theta\le\frac{1}{2}$ and let $q$ be a positive integer. Then there exists a set $W$ of cardinality at most $(e/\theta)^{q-1}$, consisting of tuples $c_1,\ldots,c_q$ of nonnegative reals with $c_1+\cdots+c_q=1$, with the following property: for every set of reals $A_1,\ldots,A_q$ and $\Lambda$ with $A_j\le 0$ for $j=1,\ldots,q$ and $\sum_{j=1}^qA_j\le -\Lambda$, there exists a tuple $(c_1,\ldots,c_q)\in W$ such that
$$A_j\le -c_j(1-\theta)\Lambda\text{ for }j=1,\ldots,q.$$
\end{lemma}

\vskip0.2cm
\noindent
\textbf{Part B. Proof of Quantitative Subspace Theorem.} 

\vskip0.1cm
\textbf{(a)} let $K$ be a number field, $S$ be a finite set of places of $K$ and $X,N,n,d,s,C,f^{(v)}_i$ $(v\in S,i=0,\ldots,n),C,\Delta,A_1,A_2,A_3,H$ be as in Theorem \ref{1.5}. We denote the coordinates on $\P^N$ by ${\bf x}=(x_0,\ldots,x_N).$

Let $f_0,\ldots,f_R$ be the distinct polynomials among $\sigma(f_{j}^{(v)})\ (v\in S,j=0,\ldots,n,\sigma\in G_K)$. From (\ref{1.6}), we have
\begin{align}\label{3.16}
R\le C(n+1)s-1.
\end{align}
Denote by $K'$ the extension of $K$ generated by all coefficients of $f_0,\ldots,f_R$. Put $g_i=f_i^{\Delta/\deg f_i}$ for $i=0,\ldots, R$. Then $g_0,\ldots,g_R$  are homogenenous polynomials in $K[x_0,\ldots,x_u]$ of degree $\Delta$. We define
$$\varphi : {\bf x}\mapsto (g_0({\bf x}),\ldots,g_R({\bf x})), Y:=\varphi(X).$$
From assumption (\ref{1.8}), $\varphi$ is a finite morphism on $X$, and $Y$ is a projective subvariety of $\P^R$ defined over $K'$. We have
\begin{align}\label{3.17}
\dim Y=n, \deg Y=: D\le d\Delta^n.
\end{align}
We denote places on $K'$ by $v'$ and define normalized absolute values $|.|_{v'}$ on $K'$ similarly to Section 1. For every $v'\in M_K$, we choose an extension of $|.|_{v'}$ to ${\overline{\Q}}$. Since $K'/K$ is a normal extension, for every $v'\in M_{K'}$, there is $\tau_{v'}\in G_K$ such that
\begin{align}\label{3.18}
|x|_{v'} =|\tau_{v'}(x)|^{1/g(v)}_v\text{ for }x\in{\overline{\Q}}
\end{align}
where $v\in M_K$ is the place below $v'$ and $g(v)$ is the number of places of $K'$ lying above $v$. For each $v'\in M^\infty_{K'}$, there is an isomorphic embedding $\sigma_{v'}:K'\hookrightarrow\C$ such that $|x|_{v'}=|\sigma_{v'}(x)|^{[K_{v'}:\R]/[K':\Q]}$ for $x\in{\overline{\Q}}$. Then, we define norms $||.||_{v'}$,$||.||_{v',1}$ for polynomials similarly as in (\ref{3.11}), with $K',v',\sigma_{v'}$ in place of $K,v,\sigma_v.$

\vskip0.2cm
\noindent
\vskip0.1cm
\textbf{(b)} We now give an above bound for $h_1(1,g_0,\ldots,g_R)$ and $h(Y)$ as follows. Firstly, for $v'\in M^\infty_{K'}$, we have 
\begin{align*}
||1,\sigma_{v'}(g_0),\ldots,\sigma_{v'}(g_R)||_1&=1+\sum_{i=0}^R||\sigma_{v'}(g_i)||_1\le 1+\sum_{i=0}^R||\sigma_{v'}(f_i)||_1^{\Delta/\deg f_i}\\ 
&\le 1+\sum_{i=0}^R\left (\binom{\deg f_i+N}{\deg f_i}||\sigma_{v'}(f_i)||\right)^{\Delta/\deg f_i}\\
&\le(R+2)(N+\Delta)^\Delta ||1,\sigma_{v'}(g_0),\ldots,\sigma_{v'}(g_R)||^\Delta\\
&\le (R+2)(N+\Delta)^\Delta\prod_{i=0}^R||1,\sigma_{v'}(f_i)||^{\Delta}.
\end{align*}
Then for $v'\in M^\infty_{K'}$, we have
$$||1,g_0,\ldots,g_R||_{v',1}\le\left((R+2)(N+\Delta)^\Delta\right)^{[K'_{v'}:\R]/[K':\Q]}\cdot\prod_{i=0}^R||1,f_i||^{\Delta}_{v'}.$$
On the other hand, for $v'\in M^0_{K'}$, we have
$$||1,g_0,\ldots,g_R||_{v',1}\le \prod_{i=0}^R||1,f_i||^{\Delta}_{v'}.$$
Taking the product over $v'\in M_{K'}$, using (\ref{3.16}) and noting that polynomials with conjugate sets of coefficients have the same height, we get
\begin{align*}
h_1(1,g_0,\ldots,g_R)&\le\Delta\left(\sum_{i=0}^Rh(1,f_i)\right)+\Delta\log\left((R+2)(N+\Delta)^\Delta\right)\\ 
&\le\Delta C\left(\sum_{v\in S}\sum_{j=0}^nh(1,f^{(v)}_j)\right)+\Delta\log (N+\Delta)+\log (3Cns).
\end{align*}
Now, inserting this estimate into Proposition \ref{3.14}, we get
\begin{align*}
h(Y)\le&\Delta^{n+1}h(X)+(n+1)d\Delta^{n+1}C\sum_{v\in S}\sum_{j=0}^nh(1,f^{(v)}_j)\\
&6(n+1)d\Delta^{n+1}\log(N+\Delta)+4(n+1)d\Delta^n\log(3Cns).
\end{align*}
By an easy computation, we obtain
\begin{align}\label{3.19}
&h_1(1,g_0,\ldots,g_R)\le 6\Delta^2 Cns\cdot H,\\ 
\label{3.20}
& h(Y)\le 25n^2d\Delta^{n+2}Cs.H,
\end{align}
where $H$ is defined by (\ref{1.7}).

\vskip0.1cm
\textbf{(c)}
Fix a solution ${\bf x}\in X({\overline{\Q}})$ of (\ref{1.9}). For $v\in S$, denote by $I_v$ the subset of $\{0,\ldots,R\}$ such that $\{f^{(v)}_j : j=0,\ldots,n\}=\{f_j : i\in I_v\}$. Put $G_v:=||1,g_0,\ldots,g_R||_{v,1}$ for $v\in S.$ Then
$$\sum_{v\in S}\sum_{i\in I_v}\log\left (\max\limits_{\sigma\in G_K}\frac{|g_i({\bf x})|_v}{G_v||\sigma({\bf x})||^{\Delta}_v}\right )\le -((m-n+1)(n+1)+\delta)\Delta h({\bf x}).$$
By (\ref{3.12}), the all terms in the sum are $\le 0$. We apply Lemma \ref{4.1} with $q=(m+1)s$ and $\theta=\frac{\delta}{2((m-n+1)(n+1)+\delta)}=1-\frac{(m-n+1)(n+1)+\delta/2}{(m-n+1)(n+1)+\delta}$. This yieds that there is a set $W$ with
\begin{align}\label{3.21}
\begin{split}
\sharp W\le\left (\frac{2e((m-n+1)(n+1)+\delta)}{\delta}\right)^{(m+1)s-1}&\le (6e(m+n-1)n\delta^{-1})^{(m+1)s-1}\\
&\le (17(m+n-1)n\delta^{-1})^{(m+1)s-1}.
\end{split}
\end{align}
consisting of tuples of nonnegative reals $(c_{iv}\ :\ v\in S, i\in I_v)$ with
\begin{align}\label{3.22}
\sum_{v\in S}\sum_{i\in I_v}c_{iv}=1.
\end{align}
such that for every solution ${\bf x}\in X({\overline{\Q}})$ of (\ref{1.9}) there is a tuple $(c_{iv}\ :\ v\in S, i\in I_v)\in W$ with
\begin{align}\label{3.23}
\log\left(\max\limits_{\sigma\in G_K}\frac{|g_i(\sigma({\bf x}))|_v}{G_v\cdot ||\sigma({\bf x})||^{\Delta}_v}\right)\le -c_{iv}\left ((m-n+1)(n+1)+\frac{\delta}{2}\right)\Delta h({\bf x})
\end{align}
for all $v\in S, i\in I_v.$ Denote by $S'$ the set of places of $K'$ lying above the places in $S$. We may consider each element of $G_K$ as a permutation on $g_0,\ldots,g_R$. Let $v'\in S'$, $v$ be the place of $K$ lying below $v'$ and $\tau_{v'}\in G_K$ be given by (\ref{3.18}). Then we define the subset $I_{v'}\subset\{0,\ldots,R\}$ and $c_{i,v'}\ (i\in I_{v'})$ by
\begin{align}\label{3.24}
&\{g_i\ :\ i\in I_{v'}\}=\{\tau^{-1}_{v'}(g_j)\ :\ j\in I_v\}\text{ for }v'\in S',\\ 
\label{3.25}
&c_{i,v'}:=\frac{c_{jv}}{g(v)} \text{ for }v'\in S', i\in I_{v'},
\end{align}
where $j\in I_v$ is the index such that $g_i=\tau^{-1}_{v'}(g_j)$ and $g(v)$ is the number of places of $K'$ lying above $v$. Further, we put
$$ G_{v'}:=||1,g_0,\ldots,g_R||_{v',1}\text{ for }v'\in M_{K'}. $$
By (\ref{3.18}), we may rewrite (\ref{3.23}) as
\begin{align}\label{3.26}
\log\left (\max_{\sigma\in G_K}\frac{|g_i(\sigma({\bf x}))|_{v'}}{G_{v'}\cdot ||\sigma({\bf x})||^{\Delta}_{v'}}\right )\le -c_{i,v'}\left ( (m-n+1)(n+1)+\frac{\delta}{2}\right)\Delta h({\bf x})
\end{align}
for all $v'\in S',i\in I_{v'}.$ Combining (\ref{3.21}), (\ref{3.22}) we obtain a lemma analogous to Lemma 5.3 in \cite{EveFer08} as follows.
\begin{lemma}\label{3.27}
There is a set $W'$ of cardinality at most $(17(m+n-1)n\delta^{-1})^{(m+1)s-1}$, consisting of tuples of nonnegative reals $(c_{i,v'}\ :\ v'\in S', i\in I_{v'})$ with
\begin{align}\label{3.28}
\sum_{v\in S'}\sum_{i\in I_{v'}}c_{i,v'}=1,
\end{align}
with the property that for every ${\bf x}\in X({\overline{\Q}})$ with (\ref{1.9}) there is a tuple in $W'$ such that ${\bf x}$ satisfies (\ref{3.26}).
\end{lemma}
We consider the solutions of a fixed system (\ref{3.26}). Put
\begin{align}
\label{3.29}
&\text{$\bullet$ $c_{i,v'}=0$ for $v'\in S',i\in\{0,\ldots,R\}\setminus I_{v'}$ and $v'\in M_{K'}\setminus S', i=0,\ldots,R$.}\\
\nonumber
&\text{$\bullet$ $c_{v'}:=(c_{0,v'},\ldots,c_{R,v'})$ for $v'\in M_{K'}$ and ${\bf c}:=({\bf c}_{v'}:\ v'\in M_{K'})$.}
\end{align}
\noindent
Denote by ${\bf y}=(y_0,\ldots,y_R)$ the coordinates of $\P^R$. We define $H_{Q,{\bf c}}({\bf y}), E_Y({\bf c})$ similarly as (\ref{3.3}), (\ref{3.7}), respectively, but with $K'$ in place of $K$.

Using the below bounded estimate for the Chow weight (Theorem \ref{1.4}), we now prove the following lemma which is analogous to Lemma 5.4 in \cite{EveFer08} of Evertse and Ferreti. 
\begin{lemma}\label{3.30}
Let ${\bf x}\in X({\overline{\Q}})$ be a solution of (\ref{3.26}) satisfying (\ref{1.10}) and let $\sigma\in G_K$. Put
$${\bf y}:=\varphi(\sigma({\bf x})), Q:=\mathrm{exp}\left(((m-n+1)(n+1)+\frac{\delta}{2})\Delta h({\bf x})\right).$$
Then
\begin{align}\label{3.31}
H_{Q,{\bf c}}({\bf y})\le Q^{E_Y({\bf c})-\delta/2(n+2)^2}.
\end{align}
\end{lemma}
\begin{proof}
We will give a below bound for $E_Y(\c)$ as follows. 

For a place $v'\in S'$, we write $I_{v'}=\{i_0,...,i_m\}$. Since $X$ defined over $\overline{\Q}$ and $g_{i_0},...,g_{g_{i_m}}$ are conjugate over $K$ to power of $f_0^{(v)},....,f_m^{(v)}$ with a place $v\in S$ lying below $v'$, the assumption (\ref{1.8}) implies that 
$$ X(\overline{\Q})\cap\{g_{i_0}=0,...,g_{i_m}=0\}=\emptyset. $$
Since $Y=\varphi(X)$, for every $\y =(y_0,...,y_R)\in Y(\overline{\Q})$, there is $\x=(x_0,...,x_N)\in X(\overline{\Q})$ such that $y_i=g_i(\x)\ (0\le i\le R)$. It yields that
$$ Y(\overline{\Q})\cap\{y_{i_0},....,y_{i_m}\}=\emptyset. $$
By Theorem \ref{1.4}, we have
\begin{align}\label{3.32}
\begin{split}
\frac{1}{(n+1)D}e_Y(\c_{v'})&\ge\dfrac{1}{(m-n+1)(n+1)}(c_{i_0,v'}+\cdots+c_{i_m,v'})\\
&=\frac{1}{(m-n+1)(n+1)}\sum_{i\in I_{v'}}c_{i,v'}.
\end{split}
\end{align}
Note that for $v'\not\in S'$, we have $c_{i,v'}=0\ (0\le i\le R)$ and hence $e_{Y}(\c_{v'})=0$.

Therefore, summing over all $v'\in S'$ the both sides of (\ref{3.32}) we obtain
\begin{align}\label{3.33}
E_Y(\c)\ge\frac{1}{(m-n+1)(n+1)}. 
\end{align}

Now, let $\x$ be a solution of (\ref{3.26}) with (\ref{1.10}) and let $\sigma\in G_K$. By setting (\ref{3.26}) we see that $\sigma(\x)$ satisfies (\ref{3.26}) for all $v\in M_K$ and $i=0,...,R$. Put $y_i=g_i(\sigma(\x))\ (i=0,...,R)$ and $\y=(y_0,...,y_R)=\varphi(\sigma(\x))$. Let $L$ be a finite normal extension of $K'$ such that $\sigma(\x)\in X(L)$. Let $\omega\in M_L$. We take $v'$ to be the place of $K'$ lying below $\omega$. Then there exists $\tau\in G_{K'}$ such that $|x|_\omega=|\tau_{\omega}(x)|_{v'}^{d(\omega|v')}$ for $x\in L$, where $d(\omega|v')=[L_{\omega}:K'_{v'}]/[L:K']$. Putting $c_{i\omega}=d(\omega|v')c_{i,v'}$, we have
\begin{align*}
|y_i|_\omega Q^{c_{i\omega}}&=|g_i(\sigma(\x))|_\omega Q^{c_{i\omega}}=\left(|g_i(\tau_\omega(\sigma(\x)))|_{v'}Q^{c_{i,v'}}\right )^{d(\omega|v')}\\
&\le \left(G_{v'}||\tau_\omega(\sigma(\x))||_{v'}^{\Delta}\right)^{d(\omega|v')}= G_{v'}^{d(\omega|v')}||\sigma(\x)||_{v'}^{\Delta}.
\end{align*}
Thus 
$$ \max_{0\le i\le R}|y_i|_\omega Q^{c_{i\omega}}\le G_{v'}^{d(\omega|v')}||\sigma(\x)||_{v'}^{\Delta}.$$
By taking the product both sides of the above inequality over all $\omega\in M_L$ and using $h(\sigma(\x))=h(\x)$, we obtain
\begin{align}\label{3.34}
H_{Q,\c}(\y)\le \exp(h_1(1,g_0,...,g_R))Q^{1/((m-n+1)(n+1)+\delta/2)}.
\end{align}

From here and throughout the remaining proof of this section, we put $\alpha=(m+n-1)(n+1)$. By the definition of $Q$ and using the estimate (\ref{3.33}), we have
\begin{align*}
&\left (E_Y(\c)-\frac{\delta}{2(\alpha +1)^2}-\frac{1}{\alpha+\delta/2}\right)\log Q\\
&\ge\left( \frac{1}{\alpha}-\frac{\delta}{2(\alpha +1)^2}-\frac{1}{n+1+\delta/2}\right)\left(\alpha+\frac{\delta}2\right)\Delta h(\x)\\
&=\frac{\delta\left(2\alpha-\alpha\delta/2+1\right)}{2\alpha(\alpha+1)^2}\Delta h(\x)\ge\frac{3\delta}{4(\alpha+1)^2}\Delta h(\x)\\
& \ge\frac{3\delta\Delta}{4((m+n-1)(n+1)+1)^2}A_3H\ge 6\Delta^2Cns\ge h_1(1,g_0,...,g_R),
\end{align*}
where the last inequality follows from (\ref{3.19}).  
Combining the above inequality and (\ref{3.34}) we obtain
$$ H_{Q,\c}(\y)\le Q^{E_Y(\c)-\frac{\delta}{2(\alpha +1)^2}}. $$
The lemma is proved.
\end{proof}

\vskip0.1cm
\textbf{(d) Proof of Theorem \ref{1.5}.}
We apply Theorem \ref{3.9} with $K',\delta/(2(\alpha+1))^2$ in place of $K,\delta$ and, in view of (\ref{3.16}) and (\ref{3.17}), with $D\le d\Delta^n$ and $R=C(m+1)s-1.$ From (\ref{3.28}), (\ref{3.31}), we see that the conditions (\ref{3.4}), (\ref{3.5}), (\ref{3.6}) (with $K'$ in place of $K$) are satisfied. Denote by $B'_1,B'_2,B'_3$ the quantities obtained by substituting $\delta/(2(\alpha+1))^2$ for $\delta, C(m+1)s-1$ for $R$, and $d\Delta^n$ for $D$ in the quantities $B_1,B_2,B_3$, respectively, defined by (\ref{3.8}). Recall that if ${\bf x}$ satisfies (\ref{1.10}), then Lemma \ref{3.30} is applicable. Moreover,
\begin{align*}
\log Q&=\left (\alpha+\frac{\delta}{2}\right )\Delta h({\bf x})\ge A_3H\\
&=\mathrm{exp}\left (2^{6n+20}n^{2n+3}\delta^{-n-l}d^{n+2}\Delta^{n(n+2)} \log(2Cs)\right )\cdot H\\
&\ge\mathrm{exp}\left (2^{5n+4}(2(n + 2)^2\delta^{-1})^{n+l}(d\Delta^n)^{n+2}\log(4C(n + 1)s)\right )\cdot\left (26n^2d\Delta^{n+2}Cs\right )\cdot H\\
&=B'_3\cdot\left(26n^2d\Delta^{n+2}Cs\right )\cdot H\ge B'_3(h(Y)+1),
\end{align*}
where the last inequality follows from (\ref{3.20}). Hence we may apply Theorem \ref{3.9}.

Hence, Theorem \ref{3.9} and Lemma \ref{3.30} imply that there are homogeneous polynomials $F_1,\ldots, F_t\in K'[y_0,\ldots,y_R]$ not vanishing identically on $Y$, with $t\le B_1'$ and $\deg F_i\le  B'_2$ for $i=1,\ldots,t$, with the property that: for every solution ${\bf x}\in X({\overline{\Q}})$ of (\ref{3.26}) with (\ref{1.10}), there is $F_i\in\{F_l,\ldots,F_t\}$ such that $F_i(\varphi(\sigma({\bf x})))=0$ for every $\sigma\in G_K$. In fact, taking $Q=\mathrm{exp}((\alpha +\delta/2)\Delta h({\bf x}))$, from Theorem \ref{3.9}, there is $F_i$ such that $F_i({\bf y})=0$ for every ${\bf y}\in Y({\overline{\Q}})$ with $H_{Q,{\bf c}}({\bf y})\le Q^{E_Y({\bf c})-\delta/2(\alpha+1)^2}$, and then by Lemma \ref{3.30} this holds in particular for all points ${\bf y}=\varphi(\sigma({\bf x}))$.)

Therefore $\tilde F_i(\sigma({\bf x}))=0$ for $\sigma\in G_K$, where $\tilde F_i$ is the polynomial obtained by substituting $g_j$ for $y_j$ in $F_i$ for $j=0,\ldots,R$. We note that $\tilde F_i\in K'[x_0,\ldots,X_N],\deg\tilde F_i\le B'_2\Delta$ and $\tilde F_i$ does not vanish identically on $X$. Write $\tilde F_i=\sum_{k=1}^M\omega_k\tilde F_{ik}$, where $\omega_1,\ldots,\omega_M$ is a $K$-basis of $K'$, and the $\tilde F_{ik}$ are polynomials with coefficients in $K$. We may choose $G_i\in\{\tilde F_{ik}: k=1,\ldots,M\}$ not vanishing identically on $X$. Since $\sigma (\tilde F_i)({\bf x})=0$ for $\sigma\in G_K$ and the polynomials $\tilde F_{ik}$ are linear combinations of the polynomials $\sigma(\tilde F_i)\ (\sigma\in G_K)$, it follows that $\tilde F_{ik}({\bf x})=0$ for $k=1,\ldots,M$, so in particular $G_i({\bf x})=0$.

Then, there are homogeneous polynomials $G_1,\ldots,G_t\in K[x_0,\ldots,x_N]$ with $t\le B'_1$ and $\deg G_i\le B'_2\Delta$ for $i=1,\ldots,t$ not vanishing identically on $X$, such that the set of ${\bf x}\in X({\overline{\Q}})$ with (\ref{3.26}) and with (\ref{1.10}) is contained in $\bigcup_{i=1}^t(X\cap\{G_i=0\})$.

By Lemma \ref{3.27}, there are at most $T:=(17(m-n+1)n\delta^{-1})^{(m+1)s-1}$ different systems (\ref{3.26}), such that every solution ${\bf x}\in X({\overline{\Q}})$ of (\ref{1.9}) satisfies one of these systems. Therefore, there are homogeneous polynomials $G_1,\ldots,G_u\in K[x_0,\ldots,x_N]$ not vanishing identically on $X$, with $u\le B'_1T$ and with $\deg G_i\le B'_2\Delta$ for $i=1,\ldots,u,$ such that the set of ${\bf x}\in X({\overline{\Q}})$ with (\ref{1.9}), (\ref{1.10}) is contained in $\bigcup_{i=1}^u(X\cap\{G_i=0\}).$ 

In order to complete the proof of Theorem \ref{1.5}, we remain show that $B'_2\Delta=A_2$ and $B'_1T\le A_1$. Indeed, we have:
\begin{align}\nonumber
\bullet\ \  & B'_2\Delta=(4n+3)(d\Delta^n)(2(\alpha +1)^2\delta^{-1})\Delta\\
\nonumber
&\hspace{20pt}=(8n+6)((m+n-1)(n+1)+1)^2d\Delta^{n+1}\delta^{-1}=A_2,\\
\nonumber
\bullet\ \  &B'_1T\le\mathrm{exp}\left (2^{10n+4}(2((m-n+1)(n+1)+1)^2)^{2n}\delta^{-2n}(d\Delta^n)^{2n+2}\right)\\
\label{3.35}
&\hspace{20pt}\times\log(4(m+1)Cs)\log\log(4(m+1)Cs)\cdot\left(17(m-n+1)n\delta^{-1}\right)^{(m+1)s-1}.
\end{align}
We have some following fundamental estimates:
\begin{align*}
\bullet\ \ &((m-n+1)(n+1)+1)^{4n}\le (m-n+1)^{4n}(n+2)^{4n}\le (m-n+1)^{4n}n^{4n}\left(1+\frac{2}{n}\right )^{4n}\\
&\ \ \ \le (m-n+1)^{4n}n^{4n}\mathrm{exp}(8)\le 2^{12}(m-n+1)^{4n}n^{4n},\\
\bullet\ \ &\log(4(m+1)Cs)\le \sqrt{(m+1)s}\log (4C) (\text{ since $4C\ge 4$}),\\
\bullet\ \ &\log\log(4(m+1)Cs)\le\log(\sqrt{(m+1)s}\log (4C))\le 2\sqrt{(m+1)s}\log\log(4C),\\
\bullet\ \ &2(x+1)\le \frac{18^{x+1}}{17^x} \text{ for all }x>0.
\end{align*}
Then, from (\ref{3.35}) we have
\begin{align*}
B'_1T\le&\mathrm{exp}\left (2^{12n+16}(m-n+1)^{4n}n^{4n}\delta^{-2n}d^{2n+2}D^{n(2n+2)}\right)\\
&\times \left(18(m-n+1)n\delta^{-1}\right)^{(m+1)s-1}\log (4C)\log\log(4C)=A_1.
\end{align*}
The proof of the theorem is completed.
\hfill$\square$

\section{General form of Second main theorem for hypersurfaces}

In this section, we will prove Theorem \ref{1.12}. Firstly, we need to recall some following.

\vskip0.2cm
\noindent
\textbf{4.1. Nevanlinna's functions.} Let $f : \mathbb C^p \to \P^N(\C)$ be a meromorphic mapping. Let $\tilde f = ( f_0,\ldots, f_N)$ be a reduced representation of  $f,$  where $f_0,\ldots, f_N$ are holomorphic functions on $\mathbb C^p$ such that $I(f)=\{f_0=\cdots =f_N\}$ is  an analytic subset of codimension at least two of $\C^p$. The characteristic function of $f$ (with respect to the hyperplane line bundle of $\P^N(\C)$), denoted by $T_f (r)$, is defined by
$$ T_f (r)=\int_1^{r}\frac{dt}{t}\int_{B(t)}f^*\Omega\wedge v_{p-1},$$
where $B(t)=\{z\in\C^p\ ;\ ||z||<t\}$, $v_{p-1}(z) := \big(dd^c ||z||^2\big)^{p-1}$ and $\Omega$ is the Fubini-Study form on $\P^N(\C)$. By Jensen's formula, we have
$$ T_f(r)=\int_{S(r)}\log ||\tilde f||\sigma_p-\int_{S(1)}\log ||\tilde f||\sigma_p,$$
where $||\tilde f||=(|f_0|^2+\cdots |f_n|^2)^{\frac{1}{2}}$, $S(r)=\{z\in\C^p\ ;\ ||z||=r\}$ and $\sigma_p(z):= d^c \log||z||^2 \land \big(dd^c\log||z||^2\big)^{p-1}$.

Fix $(\omega_0:\cdots :\omega_N)$ be a homogeneous coordinate system on $\P^n(\C)$. Let $Q$ be a hypersurface in $\P^n(\C)$ of degree $\Delta$. Throughout this paper, we sometimes identify a hypersurface with the defining polynomial if there is no confusion.  Then we may write
$$ Q(\omega)=\sum_{I\in\mathcal T_\Delta}a_I\omega^I, $$
where $\mathcal T_\Delta=\{(i_0,...,i_N)\in\mathbb Z_{\ge 0}^{N+1}\ ;\ i_0+\cdots +i_N=\Delta\}$, $\omega =(\omega_0,...,\omega_N)$, $\omega^I=\omega_0^{i_0}...\omega_N^{i_N}$ with $I=(i_0,...,i_N)\in\mathcal T_\Delta$ and $a_I\ (I\in\mathcal T_\Delta)$ are constants, not all zeros. Here $\Z_{\ge 0}^{N+1}$ denotes the set consist of all $(N+1)$-tuples of non negative integers.  In the case $\Delta=1$, we call $Q$ a hyperplane of $\P^N(\C)$.

For each hypersurface $Q$ in $\P^N(\C)$ with $f(\C^p)\not\subset Q$, we denote by $f^*Q$ the pull back divisor of $Q$ by $f$, where $Q$ is considered as a divisor in $\P^N(\C)$. The counting function of $f$ with respect to $Q$ is denoted by $N_{Q(f)}(r)$ and defined by
$$ N_{Q(f)}(r)=\int_1^{r}\frac{dt}{t}\int_{B(t)}[f^*Q]\wedge v_{p-1},$$
where $[f^*Q]$ denotes the current generated by the divisor $f^*Q$. 

The proximity function of $f$ with respect to $Q$, denoted by $m_f (r,Q)$, is defined by
$$m_f (r,Q)=\int_{S(r)}\log\frac{||\tilde f||^\Delta}{|Q(\tilde f)|}\sigma_p-\int_{S(1)}\log\frac{||\tilde f||^\Delta}{|Q(\tilde f)|}\sigma_p,$$
where $Q(\tilde f)=Q(f_0,...,f_N)$. This definition is independent of the choice of the reduced representation of $f$. 

The first main theorem in Nevalinna theory states that
$$ \Delta T_f(r)=N_{Q(f)}(r)+m_f(r,Q)+O(1).$$

\vskip0.2cm
\noindent
\textbf{4.2. Proof of Theorem \ref{1.12}.}
By adding more hypersurfaces into the set $\{Q_1,...,Q_q\}$ if necessary, without lose of generality we may suppose that $V\cap\bigcap_{j=1}^qQ_j=\emptyset$ and that all elements $K\in\mathcal K$ has the cardinality of $m+1$ and $V\cap\bigcap_{j\in K}Q_j=\emptyset$. 

We consider the mapping $\Phi$ from $V$ into $\P^{q-1}(\C)$, which maps a point $x\in V$ into the point $\Phi(x)\in\P^{q-1}(\C)$ given by
$$\Phi(x)=(Q_1(x):\cdots : Q_2(x):\cdots:Q_{q}(x)).$$ 
Let $Y=\Phi (V)$. Since $V\cap\left (\bigcap_{j=1}^{q}Q_{j}\right )=\emptyset$, $\Phi$ is a finite morphism on $V$ and $Y$ is a complex projective subvariety of $\P^{q-1}(\C)$ with $\dim Y=k$ and $D:=\deg Y=\le \Delta^k.\deg V$. 
For every 
$${\bf a} = (a_{1},\ldots ,a_{q})\in\mathbb Z^q_{\ge 0}$$ 
and
$${\bf y} = (y_{1},\ldots ,y_{q})$$ 
we denote ${\bf y}^{\bf a} = y_{1}^{a_{1}}\ldots y_{q}^{a_{q}}$. Let $u$ be a positive integer. We set
\begin{align}\label{4.1}
n_u:=H_Y(u)-1,\ l_u:=\binom{q+u-1}{u}-1,
\end{align}
and define the space
$$ Y_u=\C[y_1,\ldots,y_q]_u/(I_Y)_u, $$
which is a vector space of dimension $n_u+1$. We fix a basis $\{v_0,\ldots, v_{n_u}\}$ of $Y_u$ and consider the meromorphic mapping $F$ with a reduced representation
$$ \tilde F=(v_0(\Phi\circ \tilde f),\ldots ,v_{n_u}(\Phi\circ \tilde f)):\C^m\rightarrow \C^{n_u+1}. $$
Hence $F$ is linearly nondegenerate, since $f$ is algebraically nondegenerate.

Now, we fix a point $z\not\in\bigcup_{i=1}^q(Q_j(\tilde f))^{-1}(0)$. We define 
$${\bf c}_z = (c_{1,z},\ldots,\ldots,c_{q,z})\in\mathbb Z^{q},$$ 
where
\begin{align}\label{4.2}
c_{j,z}:=\log\frac{||\tilde f(z)||^\Delta||Q_{j}||}{|Q_{j}(\tilde f)(z)|}\text{ for } j=1,...,q.
\end{align}
We see that $c_{j,z}\ge 0$ for all $j$. By the definition of the Hilbert weight, there are ${\bf a}_{1,z},...,{\bf a}_{H_Y(u),z}\in\mathbb N^{q}$ with
$$ {\bf a}_{i,z}=(a_{i,1,z},\ldots, a_{i,q,z})\text{ with }a_{i,s,z}\in\{1,...,l_u\}, $$
 such that the residue classes modulo $(I_Y)_u$ of ${\bf y}^{{\bf a}_{1,z}},...,{\bf y}^{{\bf a}_{H_Y(u),z}}$ form a basic of $\C[y_1,...,y_p]_u/(I_Y)_u$ and
\begin{align}\label{4.3}
S_Y(u,{\bf c}_z)=\sum_{i=1}^{H_Y(u)}{\bf a}_{i,z}\cdot{\bf c}_z.
\end{align}
We see that ${\bf y}^{{\bf a}_{i,z}}\in Y_m$ (modulo $(I_Y)_m$). Then we may write
$$ {\bf y}^{{\bf a}_{i,z}}=L_{i,z}(v_0,\ldots ,v_{H_Y(u)}), $$ 
where $L_{i,z}\ (1\le i\le H_Y(u))$ are independent linear forms.
We have
\begin{align*}
\log\prod_{i=1}^{H_Y(u)} |L_{i,z}(\tilde F(z))|&=\log\prod_{i=1}^{H_Y(u)}\prod_{1\le j\le q}|Q_{j}(\tilde f(z))|^{a_{i,j,z}}\\
&=-S_Y(m,{\bf c}_z)+\Delta uH_Y(u)\log ||\tilde f(z)|| +O(uH_Y(u)).
\end{align*}
This implies that
\begin{align*}
\log\prod_{i=1}^{H_Y(u)}\dfrac{||\tilde F(z)||\cdot ||L_{i,z}||}{|L_{i,z}(\tilde F(z))|}=&S_Y(u,{\bf c}_z)-\Delta uH_Y(u)\log ||\tilde f(z)|| \\
&+H_Y(u)\log ||\tilde F(z)||+O(uH_Y(u)).
\end{align*}
Here we note that $L_{i,z}$ depends on $i$ and $z$, but the number of these linear forms is finite. We denote by $\mathcal L$ the set of all $L_{i,z}$ occurring in the above inequalities. Then we have
\begin{align}\label{4.4}
\begin{split}
S_Y(u,{\bf c}_z)\le&\max_{\mathcal J\subset\mathcal L}\log\prod_{L\in \mathcal J}\dfrac{||\tilde F(z)||\cdot ||L||}{|L(\tilde F(z))|}+\Delta uH_Y(u)\log ||\tilde f(z)||\\
& -H_Y(u)\log ||\tilde F(z)||+O(uH_Y(u)),
\end{split}
\end{align}
where the maximum is taken over all subsets $\mathcal J\subset\mathcal L$ with $\sharp\mathcal J=H_Y(u)$ and $\{L;L\in\mathcal J\}$ is linearly independent.
From Theorem \ref{2.12new} we have
\begin{align}\label{4.5}
\dfrac{1}{uH_Y(u)}S_Y(u,{\bf c}_z)\ge&\frac{1}{(N+1)D}e_Y({\bf c}_z)-\frac{(2N+1)D}{u}\max_{1\le j\le q}c_{j,z}
\end{align}
It is clear that
\begin{align*}
\max_{1\le j\le q}c_{j,z}\le \sum_{1\le j\le q}\log\frac{||\tilde f(z)||^\Delta||Q_{j}||}{|Q_{j}(\tilde f)(z)|}+O(1),
\end{align*}
where the term $O(1)$ does not depend on $z$.
Combining (\ref{4.4}), (\ref{4.5}) and the above remark, we get
\begin{align}\nonumber
\frac{1}{(N+1)D}e_Y({\bf c}_z)\le &\dfrac{1}{uH_Y(u)}\left (\max_{\mathcal J\subset\mathcal L}\log\prod_{L\in \mathcal J}\dfrac{||\tilde F(z)||\cdot ||L||}{|L(\tilde F(z))|}-H_Y(u)\log ||\tilde F(z)||\right )\\
\label{4.6}
\begin{split}
&+\Delta\log ||\tilde f(z)||+\frac{(2n+1)D}{u}\max_{1\le j\le q}c_{j,z}+O(1/u)\\
\le &\dfrac{1}{uH_Y(u)}\left (\max_{\mathcal J\subset\mathcal L}\prod_{L\in\mathcal J}\dfrac{||\tilde F(z)||\cdot ||L||}{|L(\tilde F(z))|}-H_Y(u)\log ||\tilde F(z)||\right )\\
&+\Delta\log ||\tilde f(z)||+\frac{(2n+1)D}{u}\sum_{1\le j\le q}\log\frac{||\tilde f(z)||^\Delta||Q_{j}||}{|Q_{j}(\tilde f)(z)|}+O(1/u).
\end{split}
\end{align}
For every $1\le j_0<j_1<\cdots <j_m\le q$, we have $V\cap\bigcap_{i=0}^mQ_{j_i}=\emptyset$. Then by Lemma \ref{3.11}, we have
\begin{align*}
e_Y({\bf c}_z)\ge \frac{\Delta}{m-n+1}\cdot(c_{j_0,z}+\cdots +c_{j_m,z}).
\end{align*}
This implies that
\begin{align}\label{4.7}
e_Y({\bf c}_z)\ge \frac{D}{m-n+1}\cdot\max_{K\in\mathcal K}\log\prod_{j\in K}\frac{||\tilde f(z)||^\Delta||Q_j||}{|Q_j(\tilde f)(z)|}.
\end{align}
%where the sum is taken over all subsets $K$ of $\{1,...,q\}$ with $\sharp K=m+1$ and $V\cap\bigcap_{j\in K}Q_j=\emptyset$.
Then, from (\ref{4.6}) and (\ref{4.7}) we have
\begin{align}\label{4.8}
\begin{split}
\frac{1}{(m-n+1)(k+1)}&\max_{K}\log\prod_{j\in K}\frac{||\tilde f(z)||^\Delta||Q_j||}{|Q_j(\tilde f)(z)|}\\
&\le\dfrac{1}{uH_Y(u)}\left (\max_{\mathcal J\subset\mathcal L}\log\prod_{L\in\mathcal J}\dfrac{||\tilde F(z)||\cdot ||L||}{|L(\tilde F(z))|}-H_Y(u)\log ||\tilde F(z)||\right )\\
&+\Delta\log ||\tilde f(z)||+\frac{(2n+1)Da}{u}\sum_{1\le j\le q}\log\frac{||\tilde f(z)||^\Delta||Q_{j}||}{|Q_{j}(\tilde f)(z)|}+O(1),
\end{split}
\end{align}
where the term $O(1)$ does not depend on $z$. 

Integrating both sides of the above inequality, we obtain 
\begin{align}
\nonumber
\frac{1}{(m-n+1)(n+1)}&\int\limits_{S(r)}\max_{K\in\mathcal K}\log\prod_{j\in K}\frac{||\tilde f(z)||^\Delta||Q_j||}{|Q_j(\tilde f)(z)|}\sigma_p\\
\begin{split}\label{4.9}
\le\dfrac{1}{uH_Y(u)}&\int\limits_{S(r)}\left (\max_{\mathcal J\subset\mathcal L}\log\prod_{L\in\mathcal J}\dfrac{||\tilde F(z)||\cdot ||L||}{|L(\tilde F(z))|}-H_Y(u)\log ||\tilde F(z)||\right )\sigma_p\\
+\int\limits_{S(r)}&\left (d\log ||\tilde f(z)||+\frac{(2n+1)D}{u}\sum_{1\le j\le q}\log\frac{||\tilde f(z)||^d||Q_{j}||}{|Q_{j}(\tilde f)(z)|}\right)\sigma_p+O(1),
\end{split}
\end{align} 
On the other hand, by Theorem B (we note that Theorem B also is valid for the case of meromorphic mappings $f$ of several complex variables), for every $\epsilon'>0$ (which will be chosen later) we have
\begin{align*}
\bigl |\bigl |\ \int\limits_{S(r)}\max_{\mathcal J\subset\mathcal L}&\log\prod_{L\in\mathcal J}\dfrac{||\tilde F(z)||\cdot ||L||}{|L(\tilde F(z))|}\sigma_p-H_Y(u)T_F(r)\le\epsilon' T_F(r)=\epsilon' du T_f(r).
\end{align*}
Combining this inequality with (\ref{4.9}), we have
\begin{align}\label{4.10}
\begin{split}
\bigl |\bigl |\ \frac{1}{(m-n+1)(n+1)}&\int\limits_{S(r)}\max_{K\in\mathcal K}\log\prod_{j\in K}\frac{||\tilde f(z)||^\Delta||Q_j||}{|Q_j(\tilde f)(z)|}\sigma_p\\
&\le\dfrac{\epsilon'\Delta}{H_Y(u)}T_f(r)+\Delta T_f(r)+\frac{(2n+1)D}{u}\sum_{1\le j\le q}m_f(r,Q_i)\\
&\le \left (\Delta+\dfrac{\epsilon'\Delta}{H_Y(u)}+\frac{q(2n+1)D}{u}\right)T_f(r).
\end{split}
\end{align}
Here we note that the term $O(1)$ is absorbed by $\epsilon' \Delta u T_f(r)$.
Choosing $u$ and $\epsilon'$ such that
$$u\ge \frac{2(m-n+1)(n+1)q(2n+1)D}{\epsilon}\text{ and }\epsilon'\le \frac{2(m-n+1)(n+1)\Delta}{H_Y(u)}, $$ 
from (\ref{4.10}) we have
$$ \bigl |\bigl |\ \int\limits_{S(r)}\max_{K\in\mathcal K}\log\prod_{j\in K}\frac{||\tilde f(z)||^\Delta||Q_j||}{|Q_j(\tilde f)(z)|}\sigma_p\le (\Delta (m-n+1)(n+1)+\epsilon)T_f(r).$$
The theorem is proved. \hfill$\square$

\noindent
{\bf Acknowledgements.} This research is funded by Vietnam National Foundation for Science and Technology Development (NAFOSTED) under grant number 101.04-2018.01.

\vskip0.2cm
{\footnotesize 
\noindent
{\sc Si Duc Quang}\\
Department of Mathematics,\\
Hanoi National University of Education,\\
136-Xuan Thuy, Cau Giay, Hanoi, Vietnam.\\
\textit{E-mail}: quangsd@hnue.edu.vn

\end{document}